\documentclass[3p]{article}
\usepackage{graphicx} 
\usepackage{amssymb}
\usepackage{amsthm}
\usepackage{amsmath}
\usepackage{pgfplots}
\usepackage{CJKutf8}
\usepackage{setspace}
\usepackage{titlesec}
\allowdisplaybreaks[4]
\date{}

\theoremstyle{plain}
\newtheorem{thm}{Theorem}[section]
\newtheorem{lem}[thm]{Lemma}
\newtheorem{cor}[thm]{Corollary}
\newtheorem{prop}[thm]{Proposition}

\newtheorem*{rem}{Remark}
\newtheorem*{thmA}{Theorem A}
\newtheorem*{thmB}{Theorem B}
\newtheorem*{thmC}{Theorem C}

\newtheorem{defi}[thm]{Definition}

\begin{document}

\title{Characterization of weighted infinitesimal boundedness of Schr\"{o}dinger operator}
\author{Yanhan Chen \footnote{Department of Mathematics, Graduate School of Science, Kyoto University, Kyoto 606-8502, Japan. \newline e-mail: \texttt{chen.yanhan.67s@st.kyoto-u.ac.jp}}}
\maketitle
\renewcommand{\abstractname}{Abstract}
\begin{abstract}
    In this paper, we characterize the weighted infinitesimal boundedness: for $0<\alpha<n$ and $1<p<\infty$,
    $$\|V\phi\|_{L^{p}(w)}^{p}\leq\epsilon\|(-\Delta)^{\frac{\alpha}{2}}\phi\|_{L^{p}(w)}^{p}+C(\epsilon)\|\phi\|_{L^{p}(w)}^{p}.$$
    In particular, we extend the classical result due to Maz'ya and Verbitsky by using Carleson condition, localization estimates and capacity theory.

\end{abstract}

\textbf{Keywords:} Schr\"{o}dinger operator, Weighted estimates, Capacity

\section{Introduction}
\indent Recently, investigations into the \textit{infinitesimal boundedness}
\begin{equation}
    \|V\phi\|_{L^{p}}^{p}\leq\epsilon\|(-\Delta)^{\frac{\alpha}{2}}\phi\|_{L^{p}}^{p}+C(\epsilon)\|\phi\|_{L^{p}}^{p}\notag
\end{equation}
and the related \textit{Trudinger type inequalities} 
\begin{equation}
     \|V\phi\|_{L^{p}}^{p}\leq\epsilon\|(-\Delta)^{\frac{\alpha}{2}}\phi\|_{L^{p}}^{p}+C\epsilon^{-\beta}\|\phi\|_{L^{p}}^{p}\notag
\end{equation}
for the Schr\"{o}diner operator $H:=(-\Delta)^{\frac{\alpha}{2}}+V$ have attracted considerable attention. These two concepts play an important role in the research of self-adjointness and spectral stability of Schr\"{o}dinger operators (\cite{C80,RS75,S72}), Schr\"{o}dinger semigroup $\{e^{tH}\}$ (\cite{D97,D02,L80,SS17}) and elliptic equations (\cite{AS82,T73}). As far as we know, there is no result on weighted analysis. Specifically, we focus on the class of potential $V\in\mathcal{D}^{'}$ such that for any $\epsilon>0$, there exists $C(\epsilon)>0$ such that
\begin{equation}
    \|V\phi\|_{L^{p}(w)}^{p}\leq\epsilon\|(-\Delta)^{\frac{\alpha}{2}}\phi\|_{L^{p}(w)}^{p}+C(\epsilon)\|\phi\|_{L^{p}(w)}^{p}\quad\forall \phi\in C_{c}^{\infty},\tag{1.1}
\end{equation}
as well as the Trudinger type inequality where $C(\epsilon)$ has power growth; there exsits $C>0$ such that for any $\epsilon>0$
\begin{equation}
     \|V\phi\|_{L^{p}(w)}^{p}\leq\epsilon\|(-\Delta)^{\frac{\alpha}{2}}\phi\|_{L^{p}(w)}^{p}+C\epsilon^{-\beta}\|\phi\|_{L^{p}(w)}^{p}\quad\forall \phi\in C_{c}^{\infty},\tag{1.2}
\end{equation}
for some $\beta>0$.\\
\indent For the case $H=-\Delta+V$ with $V$ being non-negative, Kato \cite{K72} first established a sufficient condition on $V$ ensuring the inequality $(1.1)$ holds when $p=1$. Later, Zheng and Yao \cite{ZY09} extended his result to higher order cases as follows.
\begin{thmA}
    \textnormal{(Zheng and Yao \cite{ZY09})} Let $\alpha>0$ and $0\leq V\in L_{loc}^{1}$. Then the followings are equivalent.\\
\textnormal{\textnormal{(i)}} For any $\epsilon>0$, there exists $C(\epsilon)>0$ such that for any $\phi\in C_{c}^{\infty}$,
$$\|V\phi\|_{L^{1}(\mathbb{R}^{n})}\leq\epsilon\|(-\Delta)^{\frac{\alpha}{2}}\phi\|_{L^{1}(\mathbb{R}^{n})}+C(\epsilon)\|\phi\|_{L^{1}(\mathbb{R}^{n})}.$$
\textnormal{(ii)} $V\in K_{\alpha}$.
\end{thmA}
\indent Here a real-valued measurable function $V$ is said to lie in Kato Class $K_{\alpha}$ ($0<\alpha<\infty$) if 
$$\lim_{\delta\rightarrow0}\sup_{x\in\mathbb{R}^{n}}\int_{|x-y|<\delta}\omega_{\alpha}(x-y)|V(y)|dy=0\quad\text{for}\ 0<\alpha\leq n$$
and $\sup_{x\in\mathbb{R}^{n}}\int_{|x-y|<1}|V(y)|dy<\infty$ for $\alpha>n$, where
$$\omega_{\alpha}(x)=
\begin{cases}
    |x|^{\alpha-n}\quad0<\alpha<n\\
    \text{-ln}|x|\ \ \ \ \alpha=n\\
    1\quad\ \ \ \ \ \ \ \alpha>n.
\end{cases}$$
\indent In the case $p=2$, Maz'ya and Verbitsky \cite{MV05} found a complete characterization of $(1.1)$ for in $L^{2}$ by using capacity, localization estimates and a Carleson condition. Especially, in the case $V$ is a non-negative Borel measure, $V=\mu\in M^{+}(\mathbb{R}^{n})$, they have shown the following.
\begin{thmB}
\textnormal{(Maz'ya and Verbitsky \cite{MV05})} Suppose $\mu\in M^{+}(\mathbb{R}^{n})$ and $n\geq2$. Then the followings are equivalent.\\
\textnormal{(i)} For any $\epsilon>0$, there exists $C(\epsilon)>0$ such that the inequality
$$\int_{\mathbb{R}^{n}}|\phi(x)|^{2}d\mu(x)\leq\epsilon\|\nabla \phi\|^{2}_{L^{2}}+C({\epsilon})\|\phi\|^{2}_{L^{2}},\quad\forall\phi\in C_{c}^{\infty}.$$
\textnormal{(ii)}$$\lim_{\delta\rightarrow0}\sup_{x_{0}\in\mathbb{R}^{n}}\sup\left\{\int_{\mathbb{R}^{n}}|\phi(x)|^{2}d\mu(x):\phi\in C_{c}^{\infty}(B(x_{0},\delta)),\ \|\nabla\phi\|_{L^{2}}\leq1\right\}=0$$
\textnormal{(iii)}
$$\lim_{\delta\rightarrow0}\sup\left\{\frac{\mu(E)}{\textnormal{cap}(E)}:\textnormal{cap}(E)>0\ \&\ \textnormal{diam}(E)\leq\delta\right\}=0,$$
where
$$\textnormal{cap}(E):=\inf\left\{\int_{\mathbb{R}^{n}}f(x)^{2}dx:f\geq0,\ G_{1}*f\geq1\text{ on }E\right\}.$$\textnormal{
(iv)}$$\lim_{\delta\rightarrow0}\sup\left\{\frac{1}{\mu(Q)}\sum\limits_{\substack{Q^{\prime}\subseteq Q\\Q^{\prime}\in\mathcal{D}}}|Q^{\prime}|^{\frac{2}{n}-1}\mu(Q^{\prime})^{2}:Q\in\mathcal{D},\ \textnormal{diam}(Q)\leq\delta\right\}=0,$$
where $\mathcal{D}$ denotes the collection of all dyadic cubes in $\mathbb{R}^{n}.$
\end{thmB}
\indent A natural question is whether we can characterize the fractional Schr\"{o}dinger operator $H=(-\Delta)^{\alpha/2}+V$. Cao, Deng and Jin \cite{CDJ25} provided the sufficiency and necessity condition for (1.1) for the fractional Schr\"{o}dinger operator with power type potential.
\begin{thmC}
\textnormal{(Cao, Deng and Jin \cite{CDJ25})} Let $1<p<\infty$, $0<\alpha<n$ and $-n/p<a<\infty$. Then for any $\epsilon>0$, there exists $C(\epsilon)>0$ such that for any $\phi\in C_{c}^{\infty},$
\begin{align*}   \||x|^{a}\phi\|_{L^{p}}\leq\epsilon\|(-\Delta)^{\frac{\alpha}{2}}\phi\|_{L^{p}}+C(\epsilon)\|\phi\|_{L^{p}}\tag{1.3}
\end{align*}
holds if and only if $-\alpha<a\leq0.$
\end{thmC}
\indent Their proof is based on the fact that the infinitesimal boundedness (1.3) holds if and only if a specific Carleson condition similar to (iv) in Theorem B holds. They applied the sparse domination technique to control the Bessel potential $G_{\alpha,\lambda}$. Recently, Hatano, Kawasumi, Saito and Tanaka \cite{HKST25} generalized such result to the asymmetric cases $p\neq q$.\\
\indent To elaborate on the main result, we first introduce some basic notations for weighted capacity. Let $1<p<\infty$, $0<\alpha<n$ and $w$ be a weight with $\sigma=w^{1-p^{\prime}}\in A_{\infty}$. We define the weighted Riesz capacity for $E\subseteq\mathbb{R}^{n}$
\begin{equation}
R^{w}_{\alpha,p}(E):=\inf\left\{\int_{\mathbb{R}^{n}} f(x)^{p}w(x)dx:\ 0\leq f\in L^{p}(w),\ I_{\alpha}f\geq1\ on\ E\right\},\tag{1.4}
\end{equation}
where $I_{\alpha}f(x):=\int_{\mathbb{R}^{n}}|f(y)||x-y|^{\alpha-n}dy.$ Similarly, the weighted Bessel capacity for $E$ is defined by 
$$B^{w}_{\alpha,p}(E):=\inf\left\{\int_{\mathbb{R}^{n}} f(x)^{p}w(x)dx:\ 0\leq f\in L^{p}(w),\ G_{\alpha}*f\geq1\ on\ E\right\},$$
where $G_{\alpha}$ is the Bessel function. For any $f\in L^{p}(w)$, it is clear that $I_{\alpha}(G_{\alpha})*|f|\geq I_{\alpha}(G_{\alpha})*f^{+}$ ($f^{+}=(f+|f|)/2$) and $\int_{\mathbb{R}^{n}}|f(x)|^{p}w(x)dx\geq\int_{\mathbb{R}^{n}}f^{+}(x)^{p}w(x)dx$, which implies that the unnatural assumption $f\geq0$ in the definitions above can indeed be removed. Nevertheless, for the sake of simplicity, we maintain this definition below. Our main theorem is as follows.
\begin{thm}
    Let $1<p<\infty$, $0<\alpha<n$, $w\in A_{p}$ and $0\leq V\in L^{p}_{loc}\cap L^{p}_{loc}(w)\cap L^{p}_{loc}(\sigma)$ with $\sigma:=w^{1-p^{\prime}}$. For $\textnormal{(i)}\sim\textnormal{(ii)}$ below, we have $\textnormal{(i)}\leftrightarrow\textnormal{(ii)}$ and $\textnormal{(iii)}\rightarrow\textnormal{(ii)}$. Furthermore, if $\alpha<n/p$ and $\sigma\in A_{1}$, then $\textnormal{(ii)}\rightarrow\textnormal{(iii)}$.\\
\textnormal{(i)} For any $\epsilon>0$, there exists $C(\epsilon)>0$ such that for any $\phi\in C_{c}^{\infty}$, 
\begin{equation}
    \|V\phi\|_{L^{p}(w)}^{p}\leq\epsilon\|(-\Delta)^{\frac{\alpha}{2}}\phi\|_{L^{p}(w)}^{p}+C(\epsilon)\|\phi\|_{L^{p}(w)}^{p}.\notag
\end{equation}
\textnormal{(ii)} With $d\mu w:=|V|^{p}wdx$, 
\begin{equation}
\lim_{\delta\rightarrow0}\sup_{\mathcal{S}}\sup\limits_{\substack{\ell(Q)\leq\delta\\Q\in\mathcal{S}}}\left\{\frac{1}{\mu w(Q)}\sum\limits_{\substack{Q^{\prime}\in\mathcal{S}\\Q^{\prime}\subseteq Q}}|Q^{\prime}|^{p^{\prime}\left(\frac{\alpha}{n}-1\right)}\mu w(Q^{\prime})^{p^{\prime}}\sigma(Q^{\prime})\right\}=0,\tag{1.5}
\end{equation}
where the first supremum is taken oven all sparse collections $\mathcal{S}$.\\
\textnormal{(iii)}$$\lim_{\delta\rightarrow0}\sup_{\text{diam}(E)\leq\delta}\left\{\frac{\mu w(E)}{B^{w}_{\alpha,p}(E)}\right\}=0.$$
\end{thm}
\begin{rem}
     In fact, the sparse collection $\mathcal{S}$ in the condition $\textnormal{(ii)}$ can be replaced by
     \begin{equation}
\lim_{\delta\rightarrow0}\sup_{t}\sup\limits_{\substack{\ell(Q)\leq\delta\\Q\in\mathcal{D}^{t}}}\left\{\frac{1}{\mu w(Q)}\sum_{Q^{\prime}\in\mathcal{D}^{t}(Q)}|Q^{\prime}|^{p^{\prime}\left(\frac{\alpha}{n}-1\right)}\mu w(Q^{\prime})^{p^{\prime}}\sigma(Q^{\prime})\right\}=0,\tag{\textnormal{1.6}}
\end{equation}
where $\{D^{t}\}_{t\in\{0,1,2\}^{n}}$ are the families of shifted dyadic cubes. It is clear that the summation in $(1.5)$ can be bounded by the summation in $(1.6)$. Hence $\textnormal{(1.6)}\rightarrow(1.5)\rightarrow\textnormal{(i)}$. Conversely, we actually verify $\textnormal{(i)}\rightarrow(1.6)$ in the proof (see Section 3.1). Therefore, under the conditions of Theorem 1.1, \textnormal{(ii)} is also equivalent to $\textnormal{(i)}$.
\end{rem}
\indent The Bessel capacity $B^{w}_{\alpha,p}$ in $\textnormal{(iii)}$ can be replaced by the Riesz capacity $R_{\alpha,p}^{w}$ (see details in Section 3.2), even though the equivalence $B^{w}_{\alpha,p}(E)\sim R^{w}_{\alpha,p}(E)$ does not hold in general. However, as shown by Turesson \cite{T20}, for $w\in A_{p}$, the Bessel capacity $B^{w}_{\alpha,p}(E)$ can be viewed as a localized version of the Riesz capacity in the sense that
\begin{equation}
    B^{w}_{\alpha,p}(E)\sim R_{\alpha,p,1}^{w}(E),\tag{1.7}
\end{equation}
where $R_{\alpha,p,1}^{w}(E)$ is defined by (1.4) with $I_{\alpha,1}f:=\int_{|\cdot-y|\leq1}|f(y)||x-y|^{\alpha-n}dy$ in the place of $I_{\alpha}f$.
\begin{cor}
Let $1<p<\infty$, $0<\alpha<n/p$, $w$ be a weight with $\sigma:=w^{1-p^\prime}\in A_{1}$ and $0\leq V\in L^{p}_{loc}\cap L^{p}_{loc}(w)\cap L^{p}_{loc}(\sigma)$. If (1.1) holds for any $\phi\in C_{c}^{\infty}$, then
$$\lim_{\delta\rightarrow0}\sup\left\{\mu w(B(x_{0},\delta))\left(\int_{\delta}^{\infty}\frac{\sigma(B(x_{0},t))}{t^{(n-\alpha)p^{\prime}+1}}dt\right)^{p-1}:\ x_{0}\in\mathbb{R}^{n}\right\}=0.$$
\end{cor}
\indent Corollary 1.2 follows by letting $E=B(x_{0},\delta)$ in Theorem 1.1 $\textnormal{(ii)}$ with $B^{w}_{\alpha,p}$ in the place of $R^{w}_{\alpha,p}$, and
$$R^{w}_{\alpha,p}(B(x_{0},\delta))\sim\left(\int_{\delta}^{\infty}\frac{\sigma(B(x_{0},t))}{t^{(n-\alpha)p^{\prime}+1}}dt\right)^{1-p}$$
for $w\in A_{p}$, see details in Turesson \cite[Theorem 3.3.10]{T20}.\\
\indent We then give examples of power weights $w(x)=|x|^{\theta}$ and potentials $V(x)=|x|^{a}$, for which Theorem 1.1 holds. This is a weighted analogy of Theorem C by Cao, Deng and Jin \cite{CDJ25}
\begin{cor}
Under the same assumptions as in Theorem 1.1, we take power weight $w(x)=|x|^{\theta}$ with $0\leq\theta<(n-1)(p-1)$ and potential $V(x)=|x|^{a}$ satisfying $a>-n/p+(p^{\prime}-1)\theta/p$. Then the infinitesimal boundedness \textnormal{(i)} in Theorem 1.1 holds if and only if $-\alpha<a\leq0$.
\end{cor}
\indent Next we state the infinitesimal boundedness (1.1) in an equivalent localized form analogous to $\textnormal{(ii)}$ in Theorem B. For this characterization, we impose the condition that $w$ is a power weight and $\alpha$ is an integer.
\begin{thm}
Let $1\leq p<\infty$, $m\in\mathbb{N}$ and $m<n$ and $w(x)=|x|^{\theta}$ with $1-n<\theta\leq0$. Then for non-negative function $V\in L_{loc}^{1}(\mathbb{R}^{n})$, the weighted infinitesimal boundedness \textnormal{(i)} in Theorem 1.1 with $\nabla^{m}$ in the place of $(-\Delta)^{\alpha/2}$ is equivalent to
\begin{equation}
\lim_{\delta\rightarrow0}\sup_{x_{0}\in\mathbb{R}^{n}}\sup\left\{\|V\phi\|_{L^{p}(w)}:\ \phi\in C_{c}^{\infty}(B(x_{0},\delta)),\ \|\nabla^{m}\phi\|_{L^{p}(w)}\leq1\right\}=0.\tag{1.8}   
\end{equation}
\end{thm}
\begin{rem}
We note that for $m\in\mathbb{Z}$ the norms in $(1.1)$ are equivalent to classical derivatives. To be precisely, for $f\in\mathcal{S}(\mathbb{R}^{n})$, $1<p<\infty$ and $w\in A_{p}$,
$$\|\nabla^{m}f\|_{L^{p}(w)}\sim\|(-\Delta)^{\frac{m}{2}}f\|_{L^{p}(w)}.$$
However, for the endpoint case $p=1$, Theorem 1.4 applies only to the operator $H=\nabla^{m}+V$ and not to $(-\Delta)^{m/2}+V$.\\
\indent Moreover, only the sufficiency of $(1.8)$ requires $m$ to be an integer, as certain techniques relying on differential are applied. By contrast, for the necessity part, if we impose the additional assumptions $1<p<\infty$ and $\alpha<n/p$, then for $w\in A_{p,q}$, the infinitesimal boundedness property $(1.1)$ in fact implies the local behaviors $(1.8)$ (now with the general operator $(-\Delta)^{\alpha/2}$, where $\alpha$ need not be in $\mathbb{Z}^{+}$). The details are provided in the Remarks in Section 4.
\end{rem}
\indent Theorem 1.4 extends the condition $\textnormal{(ii)}$ in Theorem B and the results by Cao, Gao, Jin and Wang \cite{CGJW25} to general cases, $1\leq p<\infty$, $m\in\mathbb{Z}$ and $w$ is a power weight. Furthermore, the proof indicates an equivalent form for Trudinger type inequalities $(1.2)$ where $C(\epsilon)$ has power growth.
\begin{thm}
With the same assumption as in Theorem 1.1 except we assume $p>1$, for $\beta>0$ the followings are equivalent:\\
\textnormal{(i)} There exists $C>0$ such that for any $\epsilon>0$ and $\phi\in C_{c}^{\infty}$, 
$$\|V\phi\|_{L^{p}(w)}^{p}\leq\epsilon\|\nabla^{m}\phi\|_{L^{p}(w)}^{p}+C\epsilon^{-\beta}\|\phi\|_{L^{p}(w)}^{p}.$$
\textnormal{(ii)} There exists $C^{'}>0$ that for any $\delta>0$, $x_{0}\in\mathbb{R}^{n}$ and $\phi\in C_{c}^{\infty}(B(x_{0},\delta))$,
$$\|V\phi\|^{p}_{L^{p}(w)}\leq C^{'}\delta^{\frac{pm}{\beta+1}}\|\nabla^{m}\phi\|_{L^{p}(w)}^{p}.$$
\end{thm}
\indent The preceding theorem provides certain sufficient conditions for Trudinger type inequalities, while the subsequent deduction delves into the realm of non-linear potential theory.
\begin{cor}
Let $\beta>0$, $m\in\mathbb{N}$ with $m<n$, $1< p<(\beta+1)(n+a)/m$ and denote $q$ by
$$\frac{1}{p}-\frac{1}{q}=\frac{m}{(\beta+1)(n+a)}.$$
Assume $w(x)=|x|^{a}$ with $1-n<a\leq0$ and $0\leq V\in L^{q}(w)$. Then the Trudinger type inequalities \textnormal{(i)} in Theorem 1.4 holds if,
\begin{equation*}
v(E)\lesssim R^{w}_{m,p}(E)^{\frac{q}{p}}\ \text{for}\  E\subseteq\mathbb{R}^{n},\tag{1.9}
\end{equation*}
where $dv(x):=|V|^{q}(x)w(x)dx.$
\end{cor}

\indent This paper is organized as follows. In Section 2, we recall some basic definitions and lemmas that will be applied later. In Section 3, we prove our main result, Theorem 1.1 and its corollary Corollary 1.3. In Section 4, we provide a generalized Poincar\'{e} inequality, and then we prove Theorem 1.4,1.5 and Corollary 1.6 after that.\\
\indent We conclude the introduction with some conventions on the notation. We use $a\lesssim b$ to say that there exists a constant $C$, which is independent of the important parameters, such that $a\leq Cb$. Moreover, we write $a\sim b$ if $a\lesssim b$ and $b\lesssim a$. For any measurable set $E$, $|E|$ represents its Lebesgue measure. For a weight $w$, we set $w(E):=\int_{E}wdx$. Let $\chi_{E}$ stand for the characteristic function of $E$. In the paper all cubes are assumed to have edges parallel to the coordinate axes. For any open set $\Omega\subseteq\mathbb{R}^{n}$, $C_{c}^{\infty}(\Omega)$ denotes the space of infinitely differentiable functions with compact support on $\Omega$, and $C_{c}^{\infty}(\mathbb{R}^{n})=C_{c}^{\infty}$.
\section{Preliminary}
In this section, we recall notations and give basic estimates for Bessel potential and an operator with sparse family.
\subsection{Weight Class}
We first recall the Muckenhoupt weight \cite{M72}. For $1\leq p<\infty$ we say that $w\in A_{p}$ if
$$[w]_{A_{p}}:=\sup_{Q}\left(\frac{1}{|Q|}\int_{Q}wdx\right)\left(\frac{1}{|Q|}\int_{Q}w^{1-p^{'}}dx\right)^{p-1}<\infty,$$
where for $p=1$ we use the limiting interpretation $(\int_{Q}w^{1-p^{'}}dx/|Q|)^{p-1}=(\text{essinf}_{Q}\ w)^{-1}$. And then,
$$A_{\infty}:=\bigcup_{p\geq1}A_{p}.$$
We have althernative definition;
$$w\in A_{\infty}\leftrightarrow[w]_{A_{\infty}}:=\sup_{Q}\frac{1}{w(Q)}\int_{Q}M(w\chi_{Q})dx<\infty,$$
where $M$ denotes the Hardy-Littlewood maximal operator $Mf(x)=\sup_{Q}\int_{Q} |f|dy\cdot\chi_{Q}(x).$
This quantity $[w]_{A_{\infty}}$ is referred as the Fujii-Wilson $A_{\infty}$ constant \cite{F77,W87}. For $1\leq p,q<\infty$, following the definition by Muckenhoupt and Wheeden \cite{MW74}, we say that $w\in A_{p,q}$ if
$$[w]_{A_{p,q}}:=\sup_{Q}\left(\frac{1}{|Q|}\int_{Q}w^{q}dx\right)\left(\frac{1}{|Q|}\int_{Q}w^{-p^{'}}dx\right)^{\frac{q}{p^{\prime}}}<\infty.$$
\subsection{Bessel Functions}
\begin{defi}
Let $0<\alpha,\lambda<\infty$, the Bessel function $G_{\alpha,\lambda}$ is defined by 
$$G_{\alpha,\lambda}(x):=\mathcal{F}^{-1}\left[(2\pi)^{-\frac{n}{2}}(\lambda^{2}+|\xi|^{2})^{-\frac{\alpha}{2}}\right](x).$$
When $\lambda=1$, we simply write $G_{\alpha}(x)$. The Bessel potential $(\lambda^{2}-\Delta)^{-\alpha/2}$ is defined by 
$$(\lambda^{2}-\Delta)^{-\frac{\alpha}{2}}f(x):=G_{\alpha,\lambda}*f(x)$$
for $f\in C_{c}^{\infty}$.
\end{defi}
\indent The Bessel functions satisfy the following fundamental properties. These results can be found in standard references such as Grafakos \cite{G09}.
\begin{lem}
Let $0<\alpha,\lambda<\infty$, then\\
\textnormal{(i)} For any $x\in\mathbb{R}^{n}$, $G_{\alpha}(x)>0$ and $G_{\alpha,\lambda}(x)=\lambda^{n-\alpha}G_{\alpha}(\lambda x).$\\
\textnormal{(ii)} In the case $\alpha<n$, there exists constant $a>0$ such that 
$$G_{\alpha}(x)
\begin{cases}
\sim|x|^{\alpha-n}\quad\textnormal{if}\quad0<|x|<1,\\
\lesssim e^{-a|x|}\quad\ \textnormal{if}\quad|x|\geq1.
\end{cases}$$
\end{lem}
\indent Moreover, it is easy to verify that Bessel potentials are bounded on $L^{p}(w)$ with $w\in A_{p}$.\\
~\\
\begin{lem}
For $0<\alpha,\lambda<\infty$, $p>1$ and weight $w\in A_{p}$, it holds that
$$\|{(\lambda^{2}-\Delta)}^{-\frac{\alpha}{2}}\|_{L^{p}(w)\rightarrow L^{p}(w)}\lesssim\lambda^{-\alpha},$$
$$\|(-\Delta)^{\frac{\alpha}{2}}(\lambda^{2}-\Delta)^{-\frac{\alpha}{2}}\|_{L^{p}(w)\rightarrow L^{p}(w)}\lesssim1,$$
$$\|(\lambda^{2}-\Delta)^{\frac{\alpha}{2}}(\lambda^{2}+(-\Delta)^{\frac{\alpha}{2}})^{-1}\|_{L^{p}(w)\rightarrow L^{p}(w)}\lesssim1.$$
\end{lem}
\begin{proof}
We provide the proof for the first inequality only, because the others can be verified in the same way. By the weighted Mikhlin multiplier theorem, for example in Kurtz \cite{K80}, we obtain that for any $v\in A_{p}$, 
$$\|{(I-\Delta)}^{-\alpha/2}\|_{L^{p}(v)\rightarrow L^{p}(v)}\leq C,$$ where the constant $C$ only depends on $n,p$ and $[v]_{A_{p}}$. Furthermore, if $f_{\lambda}(x):=(1/\lambda) f(x/\lambda)$ for $\lambda>0$, we have
\begin{align*}
    (\lambda^{2}-\Delta)^{-\frac{\alpha}{2}}f(x)&=\lambda^{-\alpha}\mathcal{F}^{-1}\left[(2\pi)^{-\frac{\alpha}{2}}\left(1+\left|\frac{\xi}{\lambda}\right|^{2}\right)^{-\frac{\alpha}{2}}\widehat{f}(\xi)\right](x)\\
    &=\lambda^{-\alpha}\mathcal{F}^{-1}\left[\left[(2\pi)^{-\frac{\alpha}{2}}(1+|\cdot|^{2})^{-\frac{\alpha}{2}}\widehat{f_{\lambda}}(\cdot)\right](\frac{\xi}{\lambda})\right](x)\\
    &=\lambda^{n-\alpha}\left((I-\Delta)^{-\frac{\alpha}{2}}f_{\lambda}\right)(\lambda x).
\end{align*}
Therefore,
\begin{align*}
    \|(\lambda^{2}-\Delta)^{-\alpha/2}f\|_{L^{p}(w)}&=\lambda^{n-\alpha}\lambda^{-n/p}\|(I-\Delta)^{-\alpha/2}f_{\lambda}\|_{L^{p}\left(w(\cdot/\lambda)\right)}\\
    &\lesssim\lambda^{n-\alpha}\lambda^{-n/p}\|f_{\lambda}\|_{L^{p}\left(w(\cdot/\lambda)\right)}=\lambda^{-\alpha}\|f\|_{L^{p}(w)}.
\end{align*}
Here we used the fact that for $w\in A_{p}$ and $\lambda>0$, it holds that $w(\lambda\cdot)\in A_{p}$ and $[w(\lambda\cdot)]_{A_{p}}=[w]_{A_{p}}$.
\end{proof}
\subsection{Sparse Families}
\indent For $t\in\left\{0,1,2\right\}^{n}$, we denote the shifted dyadic cubes in $\mathbb{R}^{n}$ by 
$$\mathcal{D}^{t}:=\left\{2^{-k}\left([0,1)^{n}+m+(-1)^{k}\frac{t}{3}\right);\ k\in\mathbb{Z},m\in\mathbb{Z}^{n}\right\},\ \text{and}\ \mathcal{D}:=\bigcup_{t}\mathcal{D}^{t}.$$
For $Q_{0}\in\mathcal{D}^{t}$, we denote by $\mathcal{D}^{t}(Q_{0})$ the collection of all dyadic cubes $Q\in\mathcal{D}^{t}$ that satisfy $Q\subseteq Q_{0}$.
\begin{defi}
For $t\in\{0,1,2\}^{n}$, a collection of cubes $\mathcal{S}\subseteq\mathcal{D}^{t}$ is said to be a sparse family, if there is a pairwise disjoint collection $(E_Q)_{Q\in S}$, so that $E_Q\subseteq Q$ and $|E_Q|\geq|Q|/2$.
\end{defi}
\indent We will use the following two lemmas later.
\begin{lem}
\textnormal{(Fackler and Hyt\"{o}nen, \cite{FH17})} Let $\mu$ be a non-negative locally finite Borel measure, and we suppose $\alpha_{1}>0$ and $\alpha_{2}\geq0$ satisfy $\alpha_{1}+\alpha_{2}\geq1$. Then for any sparse family $\mathcal{S}$ and cube $Q\in\mathcal{S}$,
$$\sum\limits_{\substack{Q^{\prime}\in\mathcal{S}\\Q^{\prime}\subseteq Q}}|Q^{\prime}|^{\alpha_{1}}\mu(Q^{\prime})^{\alpha_{2}}\leq C|Q|^{\alpha_{1}}\mu(Q)^{\alpha_{2}}.$$
\end{lem}
\begin{lem}
Let $1<p<\infty$ and $\mathcal{S}$ be a sparse collection. Then for any $\{\lambda_{Q}\}_{Q\in\mathcal{S}}\subseteq[0,\infty)$ and $\sigma\in A_{\infty}$, it holds
$$\|\sum_{Q\in\mathcal{S}}\lambda_{Q}\chi_{Q}\|_{L^{p}(\sigma)}^{p}\lesssim\sum_{Q\in\mathcal{S}}\lambda_{Q}^{p}\sigma(Q).$$
\end{lem}
\begin{proof}
For any $g$ with $\|g\|_{L^{p^{\prime}}(\sigma)}=1$, one has
\begin{align*}
    \int_{\mathbb{R}^{n}}\left(\sum_{Q\in\mathcal{S}}\lambda_{Q}\chi_{Q}\right)gd\sigma&=\sum_{Q\in\mathcal{S}}\lambda_{Q}\sigma(Q)^{\frac{1}{p}}\frac{\int_{Q}gd\sigma}{\sigma(Q)}\sigma(Q)^{\frac{1}{p^{\prime}}}\\
    &\leq\left(\sum_{Q\in\mathcal{S}}\lambda_{Q}^{p}\sigma(Q)\right)^{\frac{1}{p}}\left(\sum_{Q\in\mathcal{S}}(\langle g\rangle_{Q}^{\sigma})^{p^{\prime}}\sigma(Q)\right)^{\frac{1}{p^{\prime}}},\tag{2.1}
\end{align*}
where $\langle g\rangle_{Q}^{\sigma}:=\int_{Q}gd\sigma/\sigma(Q)$. For the maximal operator $M^{\sigma}f(x):=\sup_{Q}\langle f\rangle_{Q}^{\sigma}\chi_{Q}(x)$, it is well known that $\|M^{\sigma}\|_{L^{p}(\sigma)\rightarrow L^{p}(\sigma)}<\infty$ for $1<p\leq\infty$ and any weight $\sigma$. According to the definition of sparse family, there exist disjoint sets $\{E_{Q}\}_{Q\in\mathcal{S}}$ satisfying $E_{Q}\subseteq Q$ and $|E_{Q}|\geq|Q|/2$. We remark that $\sigma(Q)\lesssim\sigma(E_{Q})$. Hence
\begin{align*}
    \int_{\mathbb{R}^{n}}\left(\sum_{Q\in\mathcal{S}}\lambda_{Q}\chi_{Q}\right)gd\sigma&\lesssim\left(\sum_{Q\in\mathcal{S}}\lambda_{Q}^{p}\sigma(Q)\right)^{\frac{1}{p}}\left(\sum_{Q\in\mathcal{S}}\inf_{E_{Q}}(M^{\sigma}g)^{p^{\prime}}\sigma(E_{Q})\right)^{\frac{1}{p^{\prime}}}\\  &\leq\left(\sum_{Q\in\mathcal{S}}\lambda_{Q}^{p}\sigma(Q)\right)^{\frac{1}{p}}\|M^{\sigma}g\|_{L^{p^{\prime}}(\sigma)}\\&\lesssim\left(\sum_{Q\in\mathcal{S}}\lambda_{Q}^{p}\sigma(Q)\right)^{\frac{1}{p}}\|g\|_{L^{p^{\prime}}(\sigma)}=\left(\sum_{Q\in\mathcal{S}}\lambda_{Q}^{p}\sigma(Q)\right)^{\frac{1}{p}},
\end{align*}
which completes the proof.
\end{proof}
\section{Proof of Theorem 1.1 and Corollary 1.3}
Here, we give a proof of Theorem 1.1 and corollary 1.3.
\subsection{Equivalence between $\textnormal{(i)}$ and $\textnormal{(ii)}$ in Theorem 1.1}
\indent We begin by establishing the connection between the infinitesimal boundedness of the Schr\"{o}dinger operator (1.1) and the asymptotic behavior of the Bessel potential (Lemma 3.1). The implication $\textnormal{(ii)}\rightarrow\textnormal{(i)}$ in Theorem 1.1 is then obtained through weighted estimates relating the Bessel potential to the local fractional maximal operator $M_{\alpha,\lambda}$ (Lemma 3.3). The opposite implication $\textnormal{(i)}\rightarrow\textnormal{(ii)}$ can be verified by testing a characteristic function directly.
\begin{lem}
Let $1<p<\infty$, $0<\alpha<n$, then for $w\in A_{p}$ and $0\leq V\in L^{p}_{loc}(w)$. Then the following are equivalent:\\
\textnormal{(i)} For any $\epsilon>0$, there exists $C(\epsilon)>0$ such that for any $\phi\in C_{c}^{\infty}$, 
$$\|V\phi\|_{L^{p}(w)}^{p}\leq\epsilon\|(-\Delta)^{\frac{\alpha}{2}}\phi\|_{L^{p}(w)}^{p}+C(\epsilon)\|\phi\|_{L^{p}(w)}^{p}.$$
\textnormal{(ii)} It holds $\lim_{\lambda\rightarrow\infty}\|V(\lambda^{2}-\Delta)^{-\alpha/2}\|_{L^{p}(w)\rightarrow L^{p}(w)}=0$.
\end{lem}
\begin{proof}
$\textnormal{(i)}\rightarrow\textnormal{(ii)}$: By applying Lemma 2.3, for any $\epsilon>0$ we see that there exists $C^{\prime}(\epsilon)>0$ that for any $\phi\in C_{c}^{\infty}$ 
\begin{align*}
    \|V(\lambda^{2}-\Delta)^{-\frac{\alpha}{2}}\phi\|_{L^{p}(w)}&\leq\epsilon\|\Delta^{\frac{\alpha}{2}}(\lambda^{2}-\Delta)^{-\frac{\alpha}{2}}\|_{L^{p}(w)}+C^{\prime}(\epsilon)\|(\lambda^{2}-\Delta)^{-\frac{\alpha}{2}}\phi\|_{L^{p}(w)}\\
    &\lesssim(\epsilon+C^{\prime}(\epsilon)\lambda^{-\alpha})\|f\|_{L^{p}(w)},
\end{align*}
which completes the proof.\\
$\textnormal{(ii)}\rightarrow\textnormal{(i)}$: We apply Lemma 2.3 to see that 
\begin{align*}
    \|V(\lambda^{2}+(-\Delta)^{\frac{\alpha}{2}})^{-1}\|_{L^{p}(w)\rightarrow L^{p}(w)}&\leq\|V(\lambda^{2}-\Delta)^{-\frac{\alpha}{2}}\|_{L^{p}(w)\rightarrow L^{p}(w)}\|(\lambda^{2}-\Delta)^{\frac{\alpha}{2}}(\lambda^{2}+(-\Delta)^{\frac{\alpha}{2}})^{-1}\|_{L^{p}(w)\rightarrow L^{p}(w)}\\
    &\lesssim\|V(\lambda^{2}-\Delta)^{-\frac{\alpha}{2}}\|_{L^{p}(w)\rightarrow L^{p}(w)}.
\end{align*}
Thus for any $\phi\in C_{c}^{\infty}$ and $\epsilon>0$, one obtains that
\begin{align*}
    \|V\phi\|_{L^{p}(w)}&\leq\|V(\lambda^{2}+(-\Delta)^{\frac{\alpha}{2}})^{-1}\|_{L^{p}(w)\rightarrow L^{p}(w)}\|(\lambda^{2}+(-\Delta)^{\frac{\alpha}{2}})\phi\|_{L^{p}(w)}\\
    &\lesssim\|V(\lambda^{2}-\Delta)^{-\frac{\alpha}{2}}\|_{L^{p}(w)\rightarrow L^{p}(w)}(\lambda^{2}\|\phi\|_{L^{p}(w)}+\|(-\Delta)^{\frac{\alpha}{2}}\phi\|_{L^{p}(w)}),
\end{align*}
which shows that \textnormal{(i)} holds true by, letting $\lambda$ be large enough.
\end{proof}
\indent Under the assumptions in Theorem 1.1, it is easy to see that for any $\lambda>0$ 
$$C_{0}(\lambda):=\|(\lambda^{2}-\Delta)^{-\frac{\alpha}{2}}(\cdot\mu)\|^{p^{\prime}}_{L^{p^{\prime}}(\mu\sigma)\rightarrow L^{p^{\prime}}(\sigma)}=\|V(\lambda^{2}-\Delta)^{-\frac{\alpha}{2}}\|^{p^{\prime}}_{L^{p}(w)\rightarrow L^{p}(w)}$$
by duality. Thus, to prove $\textnormal{(i)}\leftrightarrow\textnormal{(ii)}$ in Theorem 1.1, it is sufficient to prove the following proposition.
\begin{prop}
With the assumption in Theorem 1.1, it holds that $C_{0}(\lambda)\sim C_{1}(\lambda)$, where
$$C_{1}(\lambda):=\sup_{\mathcal{S}}\sup\limits_{\substack{\ell(Q)\leq\frac{6}{\lambda}\\Q\in\mathcal{S}}}\left\{\frac{1}{\mu w(Q)}\sum\limits_{\substack{Q^{\prime}\in\mathcal{S}\\Q^{\prime}\subseteq Q}}|Q^{\prime}|^{p^{\prime}\left(\frac{\alpha}{n}-1\right)}\mu w(Q^{\prime})^{p^{\prime}}\sigma(Q^{\prime})\right\}.$$
The first supremum is taken over all sparse families.
\end{prop}
\indent We shall prove Proposition 3.2 by using the sparse domination. To do so, we control the resolvent of the Laplacian by the local fractional maximal operator;
$$M_{\alpha,\lambda}f(x):=\sup_{0<r\leq\frac{1}{\lambda}}\frac{1}{r^{n-\alpha}}\int_{B(x,r)}|f(y)|dy.$$
\begin{lem}
Let $0<\alpha<n$, $1<p<\infty$ and $\lambda>0$, then for $w\in A_{\infty}$, it holds
$$\|(\lambda^{2}-\Delta)^{-\frac{\alpha}{2}}g\|_{L^{p}(w)}\lesssim\|M_{\alpha,\lambda}g\|_{L^{p}(w)}.$$
\end{lem}
\begin{proof}
\begin{align*}
    \|(\lambda^{2}-\Delta)^{-\frac{\alpha}{2}}g\|_{L^{p}(w)}&\leq\left[\int_{\mathbb{R}^{n}}\left(\int_{\mathbb{R}^{n}}G_{\alpha,\lambda}(x-y)|g(y)|dy\right)^{p}w(x)dx\right]^{\frac{1}{p}}\\
    &\leq\left[\int_{\mathbb{R}^{n}}\left(\int_{|x-y|\leq\frac{1}{\lambda}}G_{\alpha,\lambda}(x-y)|g(y)|dy\right)^{p}w(x)dx\right]^{\frac{1}{p}}\\
    &+\left[\int_{\mathbb{R}^{n}}\left(\int_{|x-y|>\frac{1}{\lambda}}G_{\alpha,\lambda}(x-y)|g(y)|dy\right)^{p}w(x)dx\right]^{\frac{1}{p}}=:I+II.
\end{align*}
\underline{Estimate for I}: From Lemma 2.2, we have $G_{\alpha,\lambda}(x-y)\lesssim|x-y|^{\alpha-n}$ for $|x-y|<1/\lambda$. If we denote 
$$I_{\alpha,\lambda}g(x):=\int_{|x-y|\leq\frac{1}{\lambda}}|g(y)|dy,$$
one obtains
\begin{align*}
    I\lesssim\left[\int_{\mathbb{R}^{n}}\left(\int_{|x-y|\leq\frac{1}{\lambda}}G_{\alpha,\lambda}(x-y)|g(y)|dy\right)^{p}w(x)dx\right]^{\frac{1}{p}}=\|I_{\alpha,\lambda}g\|_{L^{p}(w)}\lesssim\|M_{\alpha,\lambda}g\|_{L^{p}(w)}.
\end{align*}
Here we have used the inequality by Turesson \cite[Theorem 3.1.2]{T20}.\\
\smallskip\\
\underline{Estimate for II}: there exists $a>0$ that for $|x-y|>1/\lambda$, $G_{\alpha,\lambda}(x-y)\lesssim\lambda^{n-\alpha}e^{-a\lambda|x-y|}$. Hence,
\begin{align*}
\int_{|x-y|>\frac{1}{\lambda}}G_{\alpha}(x-y)|g(y)|dy\lesssim \lambda^{n-\alpha}\int_{|x-y|>\frac{1}{\lambda}}e^{-a\lambda|x-y|}|g(y)|dy\tag{3.1}
\end{align*}
We subdivide $\mathbb{R}^{n}$ into closed cubes $\{Q_{j}\}_{j=1}^{\infty}$ such that $\ell(Q)=1/(2\sqrt{n}\lambda)$. For $x\in\mathbb{R}^{n}$, we denote $d(x,Q):=\inf\{|x-y|:\ y\in Q\}$. By H\"{o}lder inequality, we have that
\begin{align*}
    \left(\int_{|x-y|>\frac{1}{\lambda}}G_{\alpha}(x-y)|g(y)|dy\right)^{p}&\lesssim\lambda^{p(n-\alpha)}\left(\sum_{j}\int_{Q_{j}}e^{-a\lambda|x-y|}|g(y)|dy\right)^{p}\\
    &\leq\lambda^{p(n-\alpha)}\left(\sum_{j}\int_{Q_{j}}e^{-a\lambda d(x,Q_{j})}|g(y)|dy\right)^{p}\\
    &\leq\lambda^{p(n-\alpha)}\left[\sum_{j}e^{-a\lambda d(x,Q_{j})}\left(\int_{Q_{j}}|g(y)|dy\right)^{p}\right]\left(\sum_{j}e^{-a\lambda d(x,Q_{j})}\right)^{\frac{p}{p^{\prime}}}.\tag{3.2}
\end{align*}
It is not hard to see that $\sup_{x\in\mathbb{R}^{n}}\left(\sum_{j}e^{-a\lambda d(x,Q_{j})}\right)^{\frac{p}{p^{\prime}}}<\infty$. This follows from the fact that for any fixed $k\in\mathbb{Z}^{+}$, there exists $C>0$ so that the number of cubes $Q\in\{Q_{j}\}$ satisfying $d(x,Q)\leq k/\lambda$ is bounded by $Ck^{n}$. Besides, if $x_{j}$ be the center of $Q_{j}$, then $d(x,Q_{j})\geq|x-x_{j}|-1/(4\lambda)$ for any $x\in\mathbb{R}^{n}$. We can see that
$$\int_{\mathbb{R}^{n}}e^{-a\lambda d(x,Q_{j})}\left(\int_{Q_{j}}|g(y)|dy\right)^{p}w(x)dx\lesssim\left(\int_{Q_{j}}|g(y)|dy\right)^{p}\int_{\mathbb{R}^{n}}e^{-a\lambda |x-x_{j}|}w(x)dx.$$
Moreover,
\begin{align*}
    \int_{\mathbb{R}^{n}}e^{-a\lambda |x-x_{j}|}w(x)dx&=\sum_{k=0}^{\infty}\int_{\frac{k}{\lambda}\leq|x-x_{j}|<\frac{k+1}{\lambda}}e^{-a\lambda|x-x_{j}|}w(x)dx\\
    &\leq\sum_{k=0}^{\infty}e^{-ak}w(B(x_{j},{(k+1)}/{\lambda})).
\end{align*}
Since $w\in A_{\infty}$ satisfies the doubling condition, we have
$$w(B(x_{j},(k+1)/\lambda))\lesssim\left({k+1}\right)^{nr}w(B(x_{j},1/\lambda))\lesssim(k+1)^{nr}w(Q_{j}),$$
for some $r>0$. Hence it holds
\begin{align*}
    II&\lesssim\lambda^{n-\alpha}\left(\sum_{j}\left(\int_{Q_{j}}|g(y)|dy\right)^{p}\int_{\mathbb{R}^{n}}e^{-a\lambda |x-x_{j}|}w(x)dx\right)^{\frac{1}{p}}\\
    &\lesssim\lambda^{n-\alpha}\left(\sum_{j}\left(\int_{Q_{j}}|g(y)|dy\right)^{p}w(Q_{j})\right)^{\frac{1}{p}}.\tag{3.3}
\end{align*}
If $x\in Q_{j}$, then $Q_{j}\subseteq B(x,1/\lambda)$. Thus
$$\lambda^{n-\alpha}\int_{Q_{j}}|g(y)|dy\lesssim M_{\alpha,\lambda}g(x)\quad\text{for any }x\in Q_{j},$$
hence
$$(3.3)\lesssim\left(\int_{\mathbb{R}^{n}}M_{\alpha,\lambda}g(x)^{p}w(x)dx\right)^{\frac{1}{p}}=\|M_{\alpha,\lambda}g\|_{L^{p}(w)}.$$
\end{proof}
\indent Next we follow the definition of local analogies of the Riesz Potential and fractional sparse operator introduced by Cao, Dent and Jin \cite{CDJ25}.
\begin{defi}
\textnormal{(Cao, Deng and Jin \cite{CDJ25})} Let $0<\alpha<n$ and $\lambda>0$. For $t\in\{0,1,2\}^{n}$ and a sparse family $\mathcal{S}\subseteq\mathcal{D}^{t}$, we define
$$I_{\alpha,\lambda}^{\mathcal{D}^{t}}g(x):=\sum\limits_{\substack{Q\in\mathcal{D}^{t}\\\ell(Q)\leq6/\lambda}}|Q|^{\frac{\alpha}{n}-1}\int_{Q}|g(y)|dy\cdot\chi_{Q}(x),$$
$$I_{\alpha,\lambda}^{\mathcal{S}}g(x):=\sum\limits_{\substack{Q\in\mathcal{S}\\\ell(Q)\leq6/\lambda}}|Q|^{\frac{\alpha}{n}-1}\int_{Q}|g(y)|dy\cdot\chi_{Q}(x).$$
\end{defi}
\indent It is well-known that the Riesz potential is controlled by the fractional sparse operator. Similar proposition holds for the local version.
\begin{lem}
\textnormal{(Cao, Deng and Jin \cite{CDJ25})} Given $0<\alpha<n$, $\lambda>0$, it holds
$$I_{\alpha,\lambda}g(x)\lesssim\sum_{t\in\{0,1,2\}}^{n}I_{\alpha,\lambda}^{\mathcal{D}^{t}}g(x).$$
For $t\in\{0,1,2\}^{n}$, there exists a sparse family $\mathcal{S}$ depending on $g$ and $\lambda$, so that for any $x\in\mathbb{R}^{n}$,
$$I_{\alpha,\lambda}^{\mathcal{D}^{t}}g(x)\lesssim I_{\alpha,\lambda}^{\mathcal{S}}g(x),$$
where the implicit constant is independent of $g$, $\lambda$ and $x$.
\end{lem}
\indent The local fractional sparse operator is bounded on $L^{p}$ as follows
\begin{lem}
Let $1<p^{\prime}<\infty$, $0<\alpha<n$ and $\lambda>0$, then for any $g\in C_{c}^{\infty}$ and $\mu w\in A_{\infty}$, it holds
$$\int_{\mathbb{R}^{n}}|I^{\mathcal{S}}_{\alpha,\lambda}(g\mu w)(x)|^{p^{\prime}}d\sigma(x)\lesssim C_{1}(\lambda)\int_{\mathbb{R}^{n}}|g|^{p^{\prime}}d\mu w(x),$$
where the implicit constant depends on $n$, $p$ and $[\mu w]_{A_{\infty}}$, and the constant $C_{1}(\lambda)$ is defined in Proposition 3.2.
\end{lem}
\begin{proof}
By applying Lemma 2.6 with $\lambda_{Q}=|Q|^{\alpha/n-1}\int_{Q}gd\mu w$ and the subset of $\mathcal{S}$ consisting of cubes with $\ell(Q)\leq6/\lambda$, one obtains that
\begin{align*}
    \int_{\mathbb{R}^{n}}|I_{\alpha,\lambda}^{\mathcal{S}}(g\mu w)(x)|^{p^{\prime}}&=\left\|\sum\limits_{\substack{Q\in\mathcal{S}\\\ell(Q)\leq6/\lambda}}\lambda_{Q}\chi_{Q}\right\|_{L^{p^{\prime}}(\sigma)}^{p^{\prime}}\\
    &\lesssim\sum\limits_{\substack{{Q\in\mathcal{S}}\\\ell(Q)\leq6/\lambda}}\lambda_{Q}^{p^{\prime}}\sigma(Q)\\
    &=\sum\limits_{\substack{{Q\in\mathcal{S}}\\\ell(Q)\leq6/\lambda}}|Q|^{\left(\frac{\alpha}{n}-1\right)p^{\prime}}\sigma(Q)\mu w(Q)^{p^{\prime}}(\langle g\rangle_{Q}^{\mu w})^{p^{\prime}}.\tag{3.4}
\end{align*}
Let $a_{Q}:=|Q|^{\left(\frac{\alpha}{n}-1\right)p^{\prime}}\sigma(Q)\mu w(Q)^{p^{\prime}}$, $A:=\{Q\in\mathcal{S}:\ell(Q)\leq6/\lambda\}$ and
$$T_{Q}g(x):=\frac{1}{\mu w(Q)}\int_{Q}|g(x)|d\mu w(x),\quad Q\in\mathcal{S}.$$
For any $t>0$, let $\{Q_{k}\}$ be maximal cubes in the collection $\{Q\in A:\ T_{Q}g>t\}.$ Then it holds that
\begin{align*}
    \sum\limits_{\substack{Q\in A\\T_{Q}g>t}}a_{Q}&\leq\sum_{k}\sum\limits_{\substack{Q^{\prime}\in\mathcal{S}\\Q^{\prime}\subseteq Q_{k}}}a_{Q^{\prime}}\leq C_{1}(\lambda)\sum_{k}\mu w(Q)\\
    &\leq C_{1}(\lambda)\frac{1}{t}\sum_{k}\int_{Q_{k}}|g(x)|d\mu w(x)=C_{1}(\lambda)\frac{1}{t}\int_{\mathbb{R}^{n}}|g(x)|d\mu w(x).
\end{align*}
This means that $\{T_{Q}\}_{Q\in\mathcal{S}}$ is bounded from $\ L^{1}(\mu w)$ to $l^{1,\infty}(\{a_{Q}\})$. Because the $L^{\infty}(\mu w)\rightarrow l^{\infty}(\{a_{Q}\})$ for $\{T_{Q}\}_{Q\in\mathcal{S}}$ is trivial, by interpolating them, we have that $\|\{T_{Q}\}_{Q\in\mathcal{S}}\|_{L^{p^{\prime}}(\mu w)\rightarrow l^{p^{\prime}}(\{a_{Q}\})}\lesssim C_{1}(\lambda)^{1/p^{\prime}}$. Therefore,
$$\int_{\mathbb{R}^{n}}|I_{\alpha,\lambda}^{\mathcal{S}}(g\mu w)(x)|^{p^{\prime}}\lesssim C_{1}(\lambda)\int_{\mathbb{R}^{n}}|g(x)|^{p^{\prime}}d\mu w(x),$$
which completes the proof.
\end{proof}
\indent Now we can prove the Proposition 3.2.
\begin{proof}[Proof of Proposition 3.2.]
We first prove the obvious pointwise inequality $M_{\alpha,\lambda}g(x)\lesssim I_{\alpha,\lambda}g(x)$. In fact if $0<R\leq1/\lambda$, then
\begin{align*}
    \frac{1}{R^{n-\alpha}}\int_{|x-y|\leq R}|f(y)|dy&=\frac{1}{R^{n-\alpha}}\sum_{k=0}^{\infty}\int_{\frac{R}{2^{k+1}}<|x-y|\leq\frac{R}{2^{k}}}|f(y)|dy\\
    &\leq\frac{1}{R^{n-\alpha}}\sum_{k=0}^{\infty}\left(\frac{R}{2^{k}}\right)^{n-\alpha}\int_{\frac{R}{2^{k+1}}<|x-y|\leq\frac{R}{2^{k}}}\frac{|f(y)|}{|x-y|^{n-\alpha}}dy\\
    &\sim I_{\alpha,R}f(x)\leq I_{\alpha,\lambda}f(x).
\end{align*}
\indent From Lemmas 3.3, 3.5, 3.6 above, we have
\begin{align*}
    \|(\lambda^{2}-\Delta)^{-\frac{\alpha}{2}}(g\mu w)\|^{p^{\prime}}_{L^{p^{\prime}}(\sigma)}&\lesssim\|M_{\alpha,\lambda}(g\mu w)\|_{L^{p^{\prime}}(\sigma)}^{p^{\prime}}\\
    &\lesssim\|I_{\alpha,\lambda}(g\mu w)\|_{L^{p^{\prime}}(\sigma)}^{p^{\prime}}\\
    &\lesssim\sum_{t\in\{0,1,2\}^{n}}\|I^{\mathcal{S}^{t}}_{\alpha,\lambda}(g\mu w)\|_{L^{p^{\prime}}(\sigma)}^{p^{\prime}}\\
    &\lesssim C_{1}(\lambda)\|g\|_{L^{p^{\prime}}(\mu w)}^{p^{\prime}}.
\end{align*}
By taking $g=fw^{-1}$, one gets $C_{0}(\lambda)\lesssim C_{1}(\lambda)$.\\
\indent Conversely, we shall prove $C_{1}(\lambda)\lesssim C_{0}(\lambda)$. By testing the function $h=\chi_{Q}w$ for $Q\in\mathcal{D}^{t}$ with some $t\in\{0,1,2\}^{n}$, we see that
\begin{equation*}
    C_{0}(\lambda)\geq\sup\limits_{\substack{Q\in\mathcal{D}^{t}\\\ell(Q)\leq6/\lambda}}\left\{\frac{1}{\mu w(Q)}\int_{Q}\left|(\lambda^{2}-\Delta)^{-\frac{\alpha}{2}}(\chi_{Q}\mu w)(x)\right|^{p^{\prime}}d\sigma(x)\right\}.\tag{3.5}
\end{equation*}
For any $x,y\in Q$, we have $\lambda|x-y|\leq\sqrt{n}\ell(Q)\lambda\leq6\sqrt{n}$. Then by applying Lemma 2.2, one obtains that
\begin{align*}
(\lambda^{2}-\Delta)^{-\frac{\alpha}{2}}(\chi_{Q}\mu w)(x)=\int_{Q}\lambda^{n-\alpha}G_{\lambda,\alpha}(\lambda|x-y|)d\mu w(y)\sim\int_{Q}\frac{d\mu w(y)}{|x-y|^{n-\alpha}}.
\end{align*}
The last integral is bounded from below;
$$\int_{Q}\frac{|g(y)|}{|x-y|^{n-\alpha}}\gtrsim\sum_{Q^{\prime}\in\mathcal{D}^{t}(Q)}|Q^{\prime}|^{\frac{\alpha}{n}-1}\int_{Q^{\prime}}|g(y)|d\mu w(y)\cdot\chi_{Q^{\prime}}(x).$$
The proof can be found for example in the reuslt of Moen and Cruz-Uribe \cite{MC13}. Hence
\begin{align*}
    C_{0}(\lambda)&\gtrsim\sup\limits_{\substack{Q\in\mathcal{D}^{t}\\\ell(Q)\leq6/\lambda}}\left\{\frac{1}{\mu w(Q)}\int_{Q}\left(\sum_{Q^{\prime}\in\mathcal{D}^{t}(Q)}|Q^{\prime}|^{\frac{\alpha}{n}-1}\mu w(Q)\cdot\chi_{Q^{\prime}}(x)\right)^{p^{\prime}}d\sigma(x)\right\}\\
    &\geq\sup\limits_{\substack{Q\in\mathcal{D}^{t}\\\ell(Q)\leq6/\lambda}}\left\{\frac{1}{\mu w(Q)}\sum_{Q^{\prime}\in\mathcal{D}^{t}(Q)}|Q^{\prime}|^{p^{\prime}\left(\frac{\alpha}{n}-1\right)}\mu w(Q^{\prime})^{p^{\prime}}\sigma(Q^{\prime})\right\}.
\end{align*}
According to the definition, for any sparse family $\mathcal{S}$, there exists $t\in\{0,1,2\}^{n}$ that $\mathcal{S}\subseteq\mathcal{D}^{t}$, thus the last inequality shows $C_{1}(\lambda)\lesssim C_{0}(\lambda)$, which completes the proof.
\end{proof}
\subsection{Proof of Corollary 1.3}
In this subsection, we prove Corollary 1.3 by using properties of sparse families.
\begin{proof}[Proof of Corollary 1.3.] According to the discussion in Section 3.1, the proof is reduced to establishing condition \textnormal{(ii)} in Theorem 1.1. Let $v=|x|^{\beta}$ with $\beta>-n$, and $Q(x,l)$ denote the cube centered at $x$ with side length $l$. We set $Q_{0}:=Q(0,l)$. We first estimate $v(Q)$ using a standard calculation.\\
\underline{Case 1: $2\sqrt{n}Q_{0}\cap Q=\varnothing$}: We can see that $|x_{0}|>\sqrt{n}l$. Thus for any $x\in Q$, it holds that
$$\frac{|x_{0}|}{2}<|x_{0}|-\frac{\sqrt{n}l}{2}\leq|x|\leq|x_{0}|+\frac{\sqrt{n}l}{2}<\frac{3|x_{0}|}{2},$$
one obtains
$$v(Q)=\int_{Q}|x|^{\beta}dx\sim|x_{0}|^{\beta}|Q|.$$
\underline{Case 2: $2\sqrt{n}Q_{0}\cap Q\neq\varnothing$}: In this case, it holds
$$\int_{Q}|x|^{\beta}dx\leq\int_{(2\sqrt{n}+2)Q_{0}}|x|^{\beta}dx\sim|Q|^{1+\frac{\beta}{n}}.$$
Conversely, it holds
$$\int_{Q}|x|^{\beta}dx\gtrsim|Q|^{1+\frac{\beta}{n}}$$
for any cube $Q\subseteq\mathbb{R}^{n}.$ In fact in the case $\beta\geq0$, for $Q=Q(x_{0},l)$ with $x_{0}=(a_{i})_{i=1}^{n}$, we define the sequence of points $x_{k}=(b_{i}^{k})_{i=1}^{n}$ where $b_{j}^{k}=0$ for $j\leq k$ and $b_{j}^{k}=a_{j}$ for $j>k$. By the monotonicity of $|x|^{\beta}$, we immediately obtain that
$$\int_{Q}|x|^{\beta}dx\geq\int_{Q(x_{1},l)}|x|^{\beta}dx\geq\ ...\ \geq\int_{Q({x_{n},l})}|x|^{\beta}dx=\int_{Q_{0}}|x|^{\beta}dx\sim|Q|^{1+\frac{\beta}{n}}.$$
Conversely, in the case $\beta<0$, it holds that
$$\int_{Q}|x|^{\beta}dx=\int_{Q_{0}}|x_{0}+y|^{\beta}dy\geq\int_{Q_{0}}(|x_{0}|+y)^{\beta}dy\gtrsim|Q_{0}|l^{\beta}\sim|Q|^{1+\frac{\beta}{n}}.$$
In particular, for any $Q$ with $2\sqrt{n}Q_{0}\cap Q\neq\varnothing$, it holds $v(Q)\sim|Q|^{1+\alpha/n}$.\\
\indent Now, let us first prove the sufficient part. If $-\alpha<a\leq0$, we let $C(Q)$ denote the center of cube $Q$. For any cube $Q\in\mathcal{S}$ with $\ell(Q)\leq\delta$, if $\alpha<a\leq0$, then Lemma 2.5 yields that
\begin{align*}   
\sum\limits_{\substack{Q^{\prime}\in\mathcal{S}\\Q^{\prime}\subseteq Q\\Q^{\prime}:\ type 
\ I}}|Q^{\prime}|^{p^{\prime}\left(\frac{\alpha}{n}-1\right)}\mu w(Q^{\prime})^{p^{\prime}}\sigma(Q^{\prime})&=\sum\limits_{\substack{Q^{\prime}\in\mathcal{S}\\Q^{\prime}\subseteq Q\\Q^{\prime}:\ type 
\ I}}|Q^{\prime}|^{p^{\prime}\left(\frac{\alpha}{n}-1\right)}\mu w(Q^{\prime})^{p^{\prime}-1}\sigma(Q^{\prime})\cdot\mu w(Q^{\prime})\\
&\lesssim\sum\limits_{\substack{Q^{\prime}\in\mathcal{S}\\Q^{\prime}\subseteq Q\\Q^{\prime}:\ type 
\ I}}|Q^{\prime}|^{\frac{\alpha}{n}p^{\prime}-p^{\prime}}|C(Q^{\prime})|^{(ap+\theta)(p^{\prime}-1)}|Q^{\prime}|^{p^{\prime}-1}|C(Q^{\prime})|^{(1-p^{\prime})\theta}|Q^{\prime}|\mu w(Q^{\prime})\\
&=\sum\limits_{\substack{Q^{\prime}\in\mathcal{S}\\Q^{\prime}\subseteq Q\\Q^{\prime}:\ type 
\ I}}|Q^{\prime}|^{\frac{\alpha p^{\prime}}{n}}|C(Q^{\prime})|^{ap^{\prime}}\mu w(Q^{\prime})\\
&\lesssim\sum\limits_{\substack{Q^{\prime}\in\mathcal{S}\\Q^{\prime}\subseteq Q}}|Q^{\prime}|^{\frac{(n+\alpha)p^{\prime}}{n}}\mu w(Q^{\prime})\\
&\lesssim\mu w(Q)|Q|^{\frac{(n+\alpha)p^{\prime}}{n}}.
\end{align*}
Thus one has
$$\frac{1}{\mu w(Q)}\sum\limits_{\substack{Q^{\prime}\in\mathcal{S}\\Q^{\prime}\subseteq Q\\Q^{\prime}:\ type 
\ I}}|Q^{\prime}|^{p^{\prime}\left(\frac{\alpha}{n}-1\right)}\mu w(Q^{\prime})^{p^{\prime}}\sigma(Q^{\prime})\lesssim|Q|^{\frac{(n+\alpha)p^{\prime}}{n}}\rightarrow0\quad\quad\text{as }\delta\rightarrow0.$$
Similarly, we have
\begin{align*}
    \sum\limits_{\substack{Q^{\prime}\in\mathcal{S}\\Q^{\prime}\subseteq Q\\Q^{\prime}:\ type 
\ II}}|Q^{\prime}|^{p^{\prime}\left(\frac{\alpha}{n}-1\right)}\mu w(Q^{\prime})^{p^{\prime}}\sigma(Q^{\prime})&\sim \sum\limits_{\substack{Q^{\prime}\in\mathcal{S}\\Q^{\prime}\subseteq Q\\Q^{\prime}:\ type 
\ II}}|Q^{\prime}|^{\frac{\alpha}{n}p^{\prime}-p^{\prime}}\cdot|Q^{\prime}|^{(p^{\prime}-1)(1+\frac{ap+\theta}{n})}\mu w(Q^{\prime})|Q^{\prime}|^{1+\frac{(1-p^{\prime})\theta}{n}}\\
&= \sum\limits_{\substack{Q^{\prime}\in\mathcal{S}\\Q^{\prime}\subseteq Q\\Q^{\prime}:\ type\ II}}|Q^{\prime}|^{\frac{(n+\alpha)p^{\prime}}{n}}\mu w(Q^{\prime})\\
&\lesssim\mu w(Q)|Q|^{\frac{(n+\alpha)p^{\prime}}{n}}.
\end{align*}
Thus it holds
$$\frac{1}{\mu w(Q)}\sum\limits_{\substack{Q^{\prime}\in\mathcal{S}\\Q^{\prime}\subseteq Q\\Q^{\prime}:\ type 
\ II}}|Q^{\prime}|^{p^{\prime}\left(\frac{\alpha}{n}-1\right)}\mu w(Q^{\prime})^{p^{\prime}}\sigma(Q^{\prime})\lesssim|Q|^{\frac{(n+\alpha)p^{\prime}}{n}}\rightarrow0\quad\quad\text{as }\delta\rightarrow0,$$
yields that
$$\frac{1}{\mu w(Q)}\sum\limits_{\substack{Q^{\prime}\in\mathcal{S}\\Q^{\prime}\subseteq Q}}|Q^{\prime}|^{p^{\prime}\left(\frac{\alpha}{n}-1\right)}\mu w(Q^{\prime})^{p^{\prime}}\sigma(Q^{\prime})\rightarrow0\quad\quad\text{as }\delta\rightarrow0.$$
\indent Finally, we prove the necessary part. If (ii) in Theorem 1.1 holds, then for $a>0$, choose sparse family $\mathcal{S}$ and cube $Q\in\mathcal{S}$ with $2\sqrt{n}Q_{0}\cap Q=\varnothing$, then we have
\begin{align*}
    \frac{1}{\mu w(Q)}\sum\limits_{\substack{Q^{\prime}\in\mathcal{S}\\Q^{\prime}\subseteq Q}}|Q^{\prime}|^{p^{\prime}\left(\frac{\alpha}{n}-1\right)}\mu w(Q^{\prime})^{p^{\prime}}\sigma(Q^{\prime})&\geq|Q|^{\frac{\alpha p^{\prime}}{n}-p^{\prime}}\mu w(Q)^{p^{\prime}-1}\sigma(Q)\\
    &\sim|Q|^{\frac{\alpha p^{\prime}}{n}-p^{\prime}}(|C(Q)|^{ap+\theta}|Q|)^{p^{\prime}-1}|Q||C(Q)|^{(1-p^{\prime})\theta}\\
    &=|Q|^{\frac{\alpha p^{\prime}}{n}}|C(Q)|^{ap^{\prime}},
\end{align*}
which leads to a contradiction if we let $|C(Q)|\rightarrow\infty.$ For $a\leq-\alpha$, choose a $Q\in\mathcal{S}$ with $2\sqrt{n}Q_{0}\cap Q\neq\varnothing$, it holds
\begin{align*}
    \frac{1}{\mu w(Q)}\sum\limits_{\substack{Q^{\prime}\in\mathcal{S}\\Q^{\prime}\subseteq Q}}|Q^{\prime}|^{p^{\prime}\left(\frac{\alpha}{n}-1\right)}\mu w(Q^{\prime})^{p^{\prime}}\sigma(Q^{\prime})&\geq|Q|^{\frac{\alpha p^{\prime}}{n}-p^{\prime}}\mu w(Q)^{p^{\prime}-1}\sigma(Q)\\
    &\sim|Q|^{\frac{\alpha p^{\prime}}{n}-p^{\prime}}\cdot|Q|^{(p^{\prime}-1)(1+\frac{ap+\theta}{n})}|Q|^{1+\frac{(1-p^{\prime})\theta}{n}}\\
    &=|Q|^{\frac{p^{\prime}}{n}(a+\alpha)}\rightarrow\infty\quad\quad\text{as }\delta\rightarrow0.
\end{align*}
This completes the proof.
\end{proof}
\subsection{Equivalence between $\textnormal{(ii)}$ and $\textnormal{(iii)}$ in Theorem 1.1}
\indent The implication $\textnormal{(iii)}\rightarrow\textnormal{(ii)}$ follows from standard capacity techniques. To establish $\textnormal{(ii)}\rightarrow\textnormal{(iii)}$, we proceed by classifying the dyadic cubes and handling their contributions separately in the dyadic Riesz potential. We emphasize that our proof consistently employs the Riesz potential $I_{\alpha}$ to bound the Bessel potential $G_{\alpha}$. We noted in Section 1, Theorem 1.1 remain valid when substituting the Bessel capacity $B^{w}_{\alpha,p}$ with the Riesz capacity $R^{w}_{\alpha,p}$.
\begin{proof}
For $\delta>0$, we denote 
$$k_{1}(\delta):=\sup_{t\in\{0,1,2\}^{n}}\sup\limits_{\substack{\ell(Q)\leq\delta\\Q\in\mathcal{D}^{t}}}\left\{\frac{1}{\mu w(Q)}\sum_{Q^{\prime}\in\mathcal{D}^{t}(Q)}|Q^{\prime}|^{p^{\prime}\left(\frac{\alpha}{n}-1\right)}\mu w(Q^{\prime})^{p^{\prime}}\sigma(Q^{\prime})\right\}$$
and
$$k_{2}(\delta):=\sup_{\text{diam}(E)\leq\delta}\left\{\frac{\mu w(E)}{B^{w}_{p,\alpha}(E)}\right\}.$$
We establish the equivalence $\textnormal{(iii)}\leftrightarrow\lim_{\delta\rightarrow0}k_{2}(\delta)=0$. Furthermore, by the remarks following Theorem 1.1, we obtain the corresponding equivalence $\textnormal{(ii)}\leftrightarrow\ \lim_{\delta\rightarrow0}k_{1}(\delta)\rightarrow0$.\\
\underline{$\textnormal{(iii)}\rightarrow\textnormal{(ii)}$}: for any $Q_{0}\in\mathcal{D}^{t}$ with $t\in\{0,1,2\}^{n}$, and $\ell(Q_{0})\leq\delta\ll1$, we analyze the energy function 
$$\|G_{\alpha}*(\mu w\chi_{Q})\|_{L^{p^{\prime}}(\sigma)}^{p^{\prime}}=\int_{\mathbb{R}^{n}}\left(G_{\alpha}*(\mu w\chi_{Q_{0}})\right)^{p^{\prime}}(x)d\sigma(x).$$
By duality, there exists $0\leq g\in C_{c}^{\infty}$ with $\|g\|_{L^{p}(w)}=1$, such that
$$\|G_{\alpha}*(\mu w\chi_{Q})\|_{L^{p^{\prime}}(\sigma)}^{p^{\prime}}\leq2\int_{\mathbb{R}^{n}}\left(G_{\alpha}*(\mu w\chi_{Q_{0}})\right)(x) g(x)dx=\int_{Q_{0}}G_{\alpha}*g(x)d\mu w(x).$$
For any $\lambda>0$, we denote
$$A_{g}(\lambda):=\left\{x:\ G_{\alpha}*g(x)>\lambda\right\},$$
then the condition $\textnormal{(iii)}$ yields the estimate
$$\mu w_{Q_{0}}(A_{g}(\lambda))=\mu w(Q_{0}\cap A_{g}(\lambda))\leq k_{2}(\delta)B_{p,\alpha}^{w}(Q_{0}\cap A_{g}(\lambda)),$$
where $\mu w_{Q}$ denotes the measure $\mu w_{Q}dx=\chi_{Q}d\mu w(x).$ Let $0\leq h(x):=(1/\lambda)G_{\alpha}*g(x)$, then $h(x)\geq1$ on $A_{g}(\lambda)$. By the definition of $B^{w}_{\alpha,p}$, it holds that
$$B_{p,\alpha}^{w}(Q_{0}\cap A_{g}(\lambda))\leq\frac{1}{\lambda^{p}}\int_{\mathbb{R}^{n}}g(x)^{p}w(x)dx=\frac{1}{\lambda^{p}}.$$
Hence $\|G_{\alpha}*g\|_{L^{p,\infty}(\mu w _{Q_{0}})}\leq k_{2}(\delta)^{1/p}$. Then we apply H\"{o}lder inequality to obtain
\begin{equation}
    \|G_{\alpha}*(\mu w\chi_{Q})\|_{L^{p^{\prime}}(\sigma)}^{p^{\prime}}\lesssim\|G_{\alpha}*g\|_{L^{p,\infty}(\mu w_{Q_{0}})}\|\chi_{Q_{0}}\|_{L^{p^{\prime},1}(\mu w_{Q_{0}})}\lesssim k_{2}(\delta)^{\frac{1}{p}}\mu w(Q_{0})^{\frac{1}{p^{\prime}}}.\tag{3.6}
\end{equation}
Moreover, for any $x,y\in Q_{0}$ satisfying $|x-y|\lesssim\ell(Q)\leq\delta\ll1$, Lemma 2.2 yields
$$G_{\alpha}*(\mu w_{Q_{0}})(x)=\int_{Q_{0}}G_{\alpha}(x-y)\mu w(y)dy\sim\int_{Q_{0}}\mu w(y)|x-y|^{\alpha-n}dy\geq\sum_{Q\in\mathcal{D}^{t}(Q_{0})}|Q|^{\frac{\alpha}{n}-1}\mu w(Q)\chi_{Q}(x).$$
The last inequality employs the following fundamental estimate
$$\int_{Q}\frac{|g(y)|}{|x-y|^{n-\alpha}}\gtrsim\sum_{Q^{\prime}\in\mathcal{D}^{t}(Q)}|Q^{\prime}|^{\frac{\alpha}{n}-1}\int_{Q^{\prime}}|g(y)|d\mu w(y)\cdot\chi_{Q^{\prime}}(x),$$
which was shown by Moen and Cruz-Uribe \cite{MC13}. Thus we have the lower bound,
$$\|G_{\alpha}*(\mu w\chi_{Q})\|_{L^{p^{\prime}}(\sigma)}^{p^{\prime}}\gtrsim\int_{Q_{0}}\left(\sum_{Q\in\mathcal{D}^{t}(Q_{0})}|Q|^{\frac{\alpha}{n}-1}\mu w(Q)\chi_{Q}(x)\right)^{p^{\prime}}d\sigma(x)\geq\sum_{Q\in\mathcal{D}^{t}(Q_{0})}|Q|^{\left(\frac{\alpha}{n}-1\right)p^{\prime}}\mu w(Q)^{p^{\prime}}\sigma(Q),$$
By combining this estimate with the inequality (3.6), we derive the key bound $k_{1}(\delta)\lesssim k_{2}(\delta)^{p^{\prime}/p}$, which completes the proof of $\textnormal{(ii)}$.\\
\underline{$\textnormal{(ii)}\rightarrow\textnormal{(iii)}$ for $\alpha<n/p$ and $\sigma\in A_{1}$}: We shall use an analogous discussion as we did in the proof of Lemma 3.6. With an interpolation between the weak-type bound $L^{1}(\mu w)\rightarrow l^{1,\infty}(\{a_{k}\})$ and the strong-type estimate $L^{\infty}\rightarrow l^{\infty}$, combined with Lemma 2.6, we establish that 
\begin{align*}
\int_{Q_{0}}\left(\sum\limits_{\substack{Q\in\mathcal{S}\\Q\subseteq Q_{0}}}|Q|^{\frac{\alpha}{n}-1}\int_{Q}|g(y)|d\mu w(y)\chi_{Q}(x)\right)^{p^{\prime}}\sigma(x)dx&\lesssim\sum\limits_{\substack{Q\in\mathcal{S}\\Q\subseteq Q_{0}}}|Q|^{\left(\frac{\alpha}{n}-1\right)p^{\prime}}\left(\int_{Q}|g(y)|d\mu w(y)\right)^{p^{\prime}}\sigma(Q)\\&\lesssim k_{1}(\delta)\int_{Q_{0}}|g(x)|^{p^{\prime}}d\mu w(x)\tag{3.7}
\end{align*}
for any sparse family $\mathcal{S}\subseteq\mathcal{D}^{t}$ and $Q_{0}\in\mathcal{D}^{t}$ with $\ell(Q_{0})\leq\delta\ll1$.\\
\indent We firstly show that for any $0\leq g\in L^{p^{\prime}}(\mu w)$ with supp$(g)\subseteq Q_{0}$ and $\ell(Q_{0})\leq\delta$,
\begin{equation}
    \int_{\mathbb{R}^{n}}(G_{\alpha}*(g\mu w))^{p^{\prime}}(x)d\sigma(x)\lesssim k_{1}(2\delta)\int_{\mathbb{R}^{n}}g(x)^{p^{\prime}}d\mu w(x).\tag{3.8}
\end{equation}
For each $t\in\{0,1,2\}^{n}$, there exist disjoint cubes $P_{t}^{i}\in\mathcal{D}^{t}$, ($i=1,2,...,2^{n}$), such that $Q_{0}\subseteq\bigcup_{i}P_{t}^{i}$ and $\ell(Q_{0})\leq\ell(P_{t}^{i})\leq2\ell(Q_{0})$. Furthermore, the Riesz potential is pointwisely bounded by its dyadic counterpart, which is a generalized version of Lemma 3.5; see the research by Moen and Cruz-Uribe \cite{MC13} for details. Thus 
\begin{align*}
    G_{\alpha}*(g\mu w)(x)\lesssim I_{\alpha}(g\mu w)(x)&\lesssim\sum_{t}\sum_{Q\in\mathcal{D}^{t}}|Q|^{\frac{\alpha}{n}-1}\int_{Q}g(y)d\mu w(y)\cdot\chi_{Q}(x)\\
    &\lesssim\sum_{t}\sum_{i}\sum_{Q\in\mathcal{D}^{t}}|Q|^{\frac{\alpha}{n}-1}\int_{Q}g(y)\chi_{P_{t}^{i}}(y)d\mu w(y)\cdot\chi_{Q}(x)\tag{3.9}
\end{align*}
holds from Lemma 2.2. We restrict our consideration to cubes $Q\in\mathcal{D}^{t}$ satisfying $Q\cap P_{t}^{i}\neq\varnothing$, which implies either $Q\in\mathcal{D}^{t}(P_{t}^{i})$ or $P_{t}^{i}\subseteq Q$. For the first case, it is known that there exist sparse family $\mathcal{S}\subseteq\mathcal{D}^{t}(P_{t}^{i})$ such that 
\begin{equation}
    \sum_{Q\in\mathcal{D}^{t}(P_{t}^{i})}|Q|^{\frac{\alpha}{n}-1}\int_{Q}g(y)\chi_{P_{t}^{i}}(y)d\mu w(y)\cdot\chi_{Q}(x)\lesssim\sum_{Q\in\mathcal{S}}|Q|^{\frac{\alpha}{n}-1}\int_{Q}g(y)\chi_{P_{t}^{i}}(y)d\mu w(y)\cdot\chi_{Q}(x),\tag{3.10}
\end{equation}
the detailed proof can be found in the reusult of Moen and Cruz-Uribe\cite{MC13}. Then, we apply $(3.7)$ to see
\begin{equation}
   \int_{P_{t}^{i}} \left(\sum_{Q\in\mathcal{D}^{t}(P_{t}^{i})}|Q|^{\frac{\alpha}{n}-1}\int_{Q}g(y)\chi_{P_{t}^{i}}(y)d\mu w(y)\cdot\chi_{Q}(x)\right)^{p^{\prime}}d\sigma(x)\lesssim k_{1}(\delta)\int_{P_{t}^{i}}g(x)^{p^{\prime}}d\mu w(x).\tag{3.11}
\end{equation}
For cube $Q\in\mathcal{D}^{t}$ that $P_{t}^{i}\subseteq Q$, we denote $P_{k}\in\mathcal{D}^{t}$ to be the unique cube that $P_{t}^{i}\subseteq P_{k}$ satisfying $\ell(P_{k})=2^{k}\ell({P_{t}^{i})}$ ($k=0,1,...$), then we obtain
\begin{align*}
    \sum\limits_{\substack{Q\in\mathcal{D}^{t}\\P_{t}^{i}\subseteq Q}}|Q|^{\frac{\alpha}{n}-1}\int_{Q}g(y)\chi_{P_{t}^{i}}(y)d\mu w(y)\cdot\chi_{Q}(x)&=\sum_{k}|P_{k}|^{\frac{\alpha}{n}-1}\int_{P_{t}^{i}}g(y)d\mu w(y)\cdot\chi_{P_{k}}(x)\\
    &=\int_{P_{t}^{i}}g(y)d\mu w(y)\cdot\sum_{k}\sum_{j\geq k}|P_{j}|^{\frac{\alpha}{n}-1}\chi_{P_{k}\backslash P_{k-1}}(x)\\
    &\lesssim|P_{t}^{i}|^{\frac{\alpha}{n}-1}\int_{P_{t}^{i}}g(y)d\mu w(y)\cdot\sum_{k}2^{(\alpha-n)k}\chi_{P_{k}}(x).
\end{align*}
The inequality (3.7) with the sparse family $\mathcal{S}=\{P_{t}^{i}\}$ yields
\begin{align*}
    &\int_{\mathbb{R}^{n}}\left(\sum\limits_{\substack{Q\in\mathcal{D}^{t}\\P_{t}^{i}\subseteq Q}}|Q|^{\frac{\alpha}{n}-1}\int_{Q}g(y)\chi_{P_{t}^{i}}(y)d\mu w(y)\cdot\chi_{Q}(x)\right)^{p^{\prime}}d\sigma(x)\\&\lesssim|P_{t}^{i}|^{\left(\frac{\alpha}{n}-1\right)p^{\prime}}\left(\int_{P_{t}^{i}}g(y)d\mu w(y)\right)^{p^{\prime}}\sum_{k}2^{(\alpha-n)kp^{\prime}}\sigma(P_{k})\\
    &\lesssim[\sigma]_{A_{1}}|P_{t}^{i}|^{\left(\frac{\alpha}{n}-1\right)p^{\prime}}\left(\int_{P_{t}^{i}}g(y)d\mu w(y)\right)^{p^{\prime}}\sigma(P_{t}^{i})\sum_{k}2^{(\alpha-n)kp^{\prime}}\cdot2^{kn}\\
    &\lesssim k_{1}(2\delta)\int_{P_{t}^{i}}g(x)^{p^{\prime}}d\mu w(x).\tag{3.12}
\end{align*}
Here we have used the result that for $S\subseteq Q$ and $\sigma\in A_{1}$, it holds that $\sigma(Q)/|Q|\leq[\sigma]_{A_{1}}\sigma(S)/|S|$. The inequalities (3.9), (3.11), (3.12) complete the proof of (3.7).\\
\indent Now for any $0\leq f\in L^{p}(w)$ and $Q_{0}$ with $\ell(Q_{0})\leq\delta$, there exists non-negative $g$ supported on $Q_{0}$ with $\|g\|_{L^{p^{\prime}}(\mu w)}=1$ that
\begin{align*}
    \|(G_{\alpha}*f)\chi_{Q_{0}}\|_{L^{p}(\mu w)}&\leq2\int_{Q_{0}}(G_{\alpha}*f)(x)g(x)d\mu w(x)=2\int_{Q_{0}}(G_{\alpha}*(g\mu w))(x)f(x)dx.\tag{3.13}
\end{align*}
We apply H\"{o}lder inequality and (3.8) to get
\begin{align*}
        (G_{\alpha}*f)\chi_{Q_{0}}\|_{L^{p}(\mu w)}\lesssim\|G_{\alpha}*g\mu w\|_{L^{p^{\prime}}(\sigma)}\|f\|_{L^{p}(w)}\lesssim k_{1}(2\delta)^{\frac{1}{p^{\prime}}}\|g\|_{L^{p^{\prime}}(\mu w)}\|f\|_{L^{p}(w)}=k_{1}(2\delta)^{\frac{1}{p^{\prime}}}\|f\|_{L^{p}(w)}.
\end{align*}
Let $E\subseteq\mathbb{R}^{n}$ satisfies $\text{diam}(E)\leq\delta$, and take a cube $Q_{0}$ with $\ell(Q_{0})\leq2\delta$ and $P\subseteq Q_{0}$. For any $0\leq f\in L^{p}(w)$ with $G_{\alpha}*f\geq1$ on $E$, it holds
$$\mu w(E)\leq\int_{Q_{0}}(G_{\alpha}*f)^{p}(x)d\mu w(x)\lesssim k_{1}(2\delta)^{\frac{p}{p^{\prime}}}\int_{\mathbb{R}^{n}}f(x)^{p}dw(x),$$
which leads to the estimate $\mu w(E)/B^{w}_{\alpha,p}(E)\rightarrow0$ as $\delta\rightarrow0$.
\end{proof}
\indent As we noted in Section 1, our proof of the equivalence, $\textnormal{(ii)}\leftrightarrow\textnormal{(iii)}$ relies entirely on estimating the Bessel potential $G_{\alpha}$ by the Riesz potential $I_{\alpha}$. This approach still works even if the roles of $G_{\alpha}$ and $I_{\alpha}$ are reversed. Consequently, Theorem 1.1 remains valid when the Riesz capacity $R_{\alpha,p}^{w}$ is replaced by the Bessel capacity $B^{w}_{\alpha,p}$.
\section{Proof of Theorem 1.4, 1.5 and Corollary 1.6}
\indent In this section, we verify the infinitesimal boundedness (1.1) and corresponding Trduinger's inequalities (1.2) in an equivalent localized form with power weight. We first prove a higher-order weighted Poincar\'{e} inequality on balls, which straightly leads to the necessity of (1.8). The proof of the sufficiency depends on the weighted Gagliardo-Nirenberg inequality that allows us deal with the higher-order differentiations. The proof of Trudinger type inequality (Theorem 1.4) can be directly obtained from the proof of Theorem 1.4. Finally, we provide a proof for Corollary 1.6.\\
\indent Given function $f$, we recall the spherical maximal operator defined by
$M_{S}f(x):=\sup_{t>0}\left|A_{t}f(x)\right|,$ where
$$A_{t}f(x):=\left|\int_{S^{n-1}}f(\cdot-ty)d\sigma(y)\right|$$
for $t>0$. Then the following weighted Poincar\'{e} inequality holds.
\begin{lem}
Let $1\leq p<\infty$ and $m\in\mathbb{N}$ with $m<n$, then for any $\delta>0$, $x_{0}\in\mathbb{R}^{n}$ and $u\in C_{c}^{\infty}(B(x_{0},\delta))$, it holds that
$$\|u\|_{L^{p}(B(x_{0},\delta),w)}\lesssim\delta^{m}\|\nabla^{m}u\|_{L^{p}(B(x_{0},\delta),M_{S}w)}.$$
\end{lem}
\begin{proof}
We first recall the following pointwise inequality by Zhang and Niu \cite{ZN15} that for any $x\in B(x_{0},\delta)$ and $u\in C_{c}^{\infty}(B(x_{0},\delta))$
$$|u(x)|\lesssim\int_{B(x_{0},\delta)}|x-y|^{m-n}|\nabla^{m}u(y)|dy.$$
According to Minkowski inequality,
\begin{align*}
    \|u\|_{L^{p}(B(x_{0},\delta),w)}&\lesssim\left[\int_{B(x_{0},\delta)}\left(\int_{B(x_{0},\delta)}|x-y|^{m-n}|\nabla^{m}u(y)|dy\right)^{p}w(x)dx\right]^{\frac{1}{p}}\\
    &\lesssim\int_{B(x_{0},\delta)}\left(\int_{B(x_{0},\delta)}|x-y|^{p(m-n)}w(x)dx\right)^{\frac{1}{p}}|\nabla^{m}u(y)|dy.\tag{4.1}
\end{align*}
Here,
\begin{align*}
    \int_{B(x_{0},\delta)}|x-y|^{p(m-n)}w(x)dx&\leq\int_{B(0,2\delta)}|x|^{p(m-n)}w(x+y)dx\\
    &\sim\int_{0}^{2\delta}\rho^{n-1-p(n-m)}\int_{S^{n-1}}w(y+\rho\sigma)d\sigma d\rho\\
    &\lesssim\delta^{p(m-n)+n}\left(M_{S}w(y)\right)^{\frac{1}{p}}.
\end{align*}
Thus, H\"{o}lder inequality yields
\begin{align*}
     \|u\|_{L^{p}(B(x_{0},\delta),w)}&\lesssim\delta^{m-n+\frac{n}{p}}\int_{B(x_{0},\delta)}|\nabla^{m}u(y)|\left(M_{S}w(y)\right)^{\frac{1}{p}}\\
    &\lesssim\delta^{m-n+\frac{n}{p}}|B(x_{0},\delta)|^{1-\frac{1}{p}}\left(\int_{B(x_{0},\delta)}|\nabla^{m}u(y)|^{p}M_{S}w(y)dy\right)^{\frac{1}{p}}\\
    &=\delta^{m}\|\nabla^{m}u\|_{L^{p}(B(x_{0},\delta),M_{S}w)}.
\end{align*}
\end{proof}
\indent Next we turn to prove Theorem 1.4.
\subsection{Proof of Theorem 1.4}
\begin{proof}
    Proof of the Necessity of (1.8): It is well-known that $M_{S}|x|^{\theta}\lesssim|x|^{\theta}$ if and only if $1-n<\theta\leq0$, see for example Manna \cite{M15}. Combining this with Lemma 4.1, we obtain that for power weights $w(x)=|x|^{\theta}$, $\epsilon>0$ and $\phi\in C_{c}^{\infty}$, we have
	\begin{align*}
		\|V\phi\|_{L^{p}(w)}^{p}\lesssim\epsilon\|\nabla^{m}\phi\|^{p}_{L^{p}(w)}+C(\epsilon)\|\phi\|_{L^{p}(w)}^{p}\lesssim(\epsilon+C(\epsilon)\delta^{pm})\|\nabla^{m}\phi\|_{L^{p}(w)}^{p},\tag{4.2}
	\end{align*}
	for any $\delta\leq(\epsilon/C(\epsilon))^{1/(pm)}$, we obtain $\|V\phi\|_{L^{p}(w)}^{p}\lesssim2\epsilon$.\\
\indent Proof of the Sufficiency  of (1.8): For any $\epsilon>0$, we choose $\delta>0$ so that $\|V\phi\|_{L^{p}(w)}^{p}\leq\epsilon\|\nabla^{m}\phi\|_{L^{p}(w)}^{p}$ holds for any $\phi\in C_{c}^{\infty}(B(x_{0},\delta))$. For $\phi\in C_{c}^{\infty}$, we may assume that supp $\phi\subseteq B(0,R)$ for some $R>0$. Let $\eta\in C_{c}^{\infty}(B(0,1))$ satisfy $0\leq\eta\leq1$, $\eta(x)=1$ for any $|x|<1/2$ and $|\nabla^{k}\eta|\lesssim1$ for $k=1,2,...,m$. Next, we pick $x_{i}\in\mathbb{R}^{n}$ that $\{x_{i}\}$ form a cubic lattice with the grid distance $\delta/2\sqrt{n}$, and let $\eta_{i}(x)=\eta((x-x_{i})/\delta)$. Let $\Phi(x)=\sum_{i}\eta_{i}(x)$, where the sum is taken over a finite number of $i$ such that $B(0,2R)\subseteq B(x_{i},\delta)$. Then it is easy to see that $1\leq\Phi\leq C_{1}$ on $B(0,2R)$ with the constant $C_{1}$ depending only on $n$, and $|\nabla^{k}\Phi|\lesssim\delta^{-k}$ on $B(0,2R)$ for $k=1,2,...m$. Now we define
	$$\tau_{i}(x):=\frac{\eta_{i}(x)}{\Phi(x)}.$$
	It holds that $\sum_{i}\tau_{i}=1$ on $B(0,R)$, and for any $x\in B(0,R)$, it holds that $$|\nabla^{k}\tau_{i}(x)|\lesssim\sum_{a=0}^{k}|\nabla^{a}\eta_{i}(x)||\nabla^{k-a}\frac{1}{\Phi}(x)|\lesssim\frac{1}{\delta^{k}}.$$
	Moreover, the lattice construction of $\{x_{i}\}$ ensures a uniform bounded overlap property: there exists a constant $C_{2}$, which depends on $n$ only, such that $\sum_{i}\chi_{B(x_{i},\delta)}(x)\leq C_{2}$ for any $x\in\mathbb{R}^{n}$. Thus for any $x\in B(0,R)$ we obtain that
	$$\sum_{i}|\tau_{i}(x)|^{p}=\sum_{k=1}^{C_{2}}|\tau_{i_{k}}(x)|^{p}\gtrsim\left(\sum_{i}\tau_{i}(x)\right)^{p}=1\quad\text{and}\quad\sum_{i}|\nabla^{k}\tau_{i}(x)|\lesssim\delta^{-kp}.$$
	\indent We now apply the Gagliardo-Nirenberg inequality with power weight, see the result by Leinfelder \cite{L86} for $p=1$ and Duarte and Silva \cite{DD23} for $p>1$, and Lemma 4.1 to obtain:
	\begin{align*}
		\|V\phi\|_{L^{p}(w)}^{p}&=\int_{\mathbb{R}^{n}}|V(x)\phi(x)|^{p}w(x)dx\lesssim\int_{\mathbb{R}^{n}}\sum_{i}|V(x)\phi(x)\tau_{i}(x)|^{p}w(x)dx\\
		&\leq\epsilon\sum_{i}\int_{\mathbb{R}^{n}}|\nabla^{m}(\phi\tau_{i})(x)|^{p}w(x)dx\\
		&\lesssim\epsilon\sum_{i}\sum_{k=0}^{m}\int_{\mathbb{R}^{n}}|\nabla^{k}\phi|^{p}|\nabla^{m-k}\tau_{i}(x)|^{p}w(x)dx\\
		&\lesssim\epsilon\left(\sum_{k=1}^{m}\frac{1}{\delta^{p(m-k)}}\int_{\mathbb{R}^{n}}|\nabla^{k}\phi(x)|^{p}w(x)dx+\frac{1}{\delta^{pm}}\int_{\mathbb{R}^{n}}|\phi(x)|^{p}w(x)dx\right)\\
		&\lesssim\epsilon\left(\sum_{k=1}^{m}\frac{1}{\delta^{p(m-k)}}\|\nabla^{m}\phi\|_{L^{p}(w)}^{\frac{kp}{m}}\|\phi\|_{L^{p}(w)}^{p\left(1-\frac{k}{m}\right)}+\frac{1}{\delta^{pm}}\|\phi\|_{L^{p}(w)}^{p}\right)\\
        &\lesssim\epsilon\left(\sum_{k=1}^{m}\frac{1}{\delta^{p(m-k)}}\|\nabla^{m}\phi\|^{\frac{kp}{m}}_{L^{p}(w)}\delta^{pm\left(1-\frac{k}{m}\right)}\|\nabla^{m}\phi\|_{L^{p}(w)}^{p\left(1-\frac{k}{m}\right)}+\delta^{-pm}\|\phi\|_{L^{p}(w)}^{p}\right)\\
        &\lesssim\epsilon\|\nabla^{m}\phi\|^{p}_{L^{p}(w)}+\epsilon\delta^{-pm}\|\phi\|_{L^{p}(w)}^{p},\tag{4.3}
	\end{align*}
    which completes the proof.
    \end{proof}
\begin{rem}
The necessity argument only relies on the Poincaré inequality Lemma 4.1. In fact, this local estimate remains valid when even $\nabla^{m}$ is replaced by $(-\Delta)^{\alpha/2}$, provided we impose appropriate parameter constraints. To be precise, for $1<p<\infty$, $0<\alpha<n/p$ and $w$ with $w^{1/p}\in A_{p,q}$, the inequality holds that 
$$\|(-\Delta)^{-\frac{\alpha}{2}}f\|_{L^{q}(w^{\frac{q}{p}})}\lesssim\|f\|_{L^{p}(w)},$$
for $1/p-1/q=\alpha/n$, see details for example in the research by Muckenhoupt and Wheeden \textnormal{\cite{MW74}}. In combination with H\"{o}lder inequality, we can see that for any $u\in C_{c}^{\infty}(B(x_{0},\delta))$,
$$\|u\|_{L^{p}(B(x_{0},\delta),w)}=\|uw^{\frac{1}{p}}\|_{L^{p}(B(x_{0},\delta))}\leq|B(x_{0},\delta)|^{\frac{\alpha}{n}}\|uw^{\frac{1}{p}}\|_{L^{q}(B(x_{0},\delta))}\lesssim\delta^{\alpha}\|(-\Delta)^{\frac{\alpha}{2}}u\|_{L^{p}(B(x_{0},\delta),w)}.$$
An analogous argument to our previous necessity proof establishes the corresponding result in this generalized setting.
\end{rem}
\subsection{Proof of Theorem 1.5 and Corollary 1.6}
    \indent Theorem 1.5 can be indicated from the proof above.
    \begin{proof}[Proof of Themrem 1.5.]
$\textnormal{(i)}\rightarrow\textnormal{(ii)}$: Taking $C(\epsilon)\sim\delta^{-\beta}$ and $\epsilon=\delta^{\frac{pm}{\beta+1}}$ in (4.2) yields the desired condition $\textnormal{(ii)}$. $\textnormal{(ii)}\rightarrow\textnormal{(i)}:$ This follows by choosing $\delta\sim\epsilon^{\frac{\beta+1}{pm}}$ (as in the sufficiency proof), which transforms (4.3) into
$$(4.3)=\epsilon\|\nabla^{m}\phi\|^{p}_{L^{p}(w)}+\epsilon^{-\beta}\|\phi\|_{L^{p}(w)}^{p},$$
thereby leads to $\textnormal{(i)}$.
\end{proof}
\indent Finally, we prove the Corollary 1.6.
\begin{proof}[Proof of Corollary 1.6]
By Adams \cite{A86} and $(1.9)$, we derive
\begin{equation}
    \|V\phi\|_{L^{q}(w)}\lesssim\|\nabla^{m}\phi\|_{L^{p}(w)}.\tag{4.4}
\end{equation}
Moreover, by applying H\"{o}lder inequality, we can see that  
$$\|V\phi\|_{L^{p}(w)}\leq\left(w(B(x_{0},\delta))\right)^{\frac{1}{p}-\frac{1}{q}}\|V\phi\|_{L^{q}(w)}$$
for any $\phi\in C_{c}^{\infty}(B(x_{0},\delta))$. Based on the discussion presented in the proof of Corollary 1.3 in Section 3, we obtain that for both cases 1 and 2, $w(B(x_{0},\delta))\lesssim\delta^{n+a}$ holds when $w(x)=|x|^{a}$ $(1-n<a\leq0)$. Hence
$$\|V\phi\|_{L^{p}(w)}\lesssim\delta^{\left(\frac{1}{p}-\frac{1}{q}\right)(n+a)}\|V\phi\|_{L^{q}(w)}\lesssim\delta^{\frac{m}{\beta+1}}\|V\phi\|_{L^{p}(w)}.$$
Combining this with Theorem 1.4 completes the proof.
\end{proof}
\textbf{Acknowledgement}\\
\indent I would like to express my sincere gratitude to Professor J. Cao and Professor H. Tanaka for generously sharing with me their inspiring manuscripts \cite{CDJ25,HKST25} prior to their publication. I am deeply indebted to my advisor, Professor Y. Tsutsui, for his patient guidance and numerous insightful suggestions that greatly improved this work.

\end{document}